\documentclass[11pt]{amsart}

\usepackage{amsmath,amssymb,amsthm}
\usepackage[margin=1.1in]{geometry}
\usepackage{microtype}
\usepackage{xcolor}
\usepackage{hyperref}
\hypersetup{colorlinks=true,linkcolor=blue!50!black,citecolor=blue!50!black,urlcolor=blue!50!black}

\newtheorem{theorem}{Theorem}[section]
\newtheorem{lemma}[theorem]{Lemma}
\newtheorem{proposition}[theorem]{Proposition}
\newtheorem{corollary}[theorem]{Corollary}
\newtheorem{question}[theorem]{Question}
\theoremstyle{definition}
\newtheorem{definition}[theorem]{Definition}
\newtheorem{remark}[theorem]{Remark}

\newcommand{\bbox}{\mathcal{B}}
\newcommand{\NN}{\mathbb{N}}
\newcommand{\eqdef}{\mathrel{\mathop:}=}

\begin{document}

\title[Bounds for double covers of discrete boxes]
{New bounds for double covers of the discrete box $\{0,1,2\}^d$}

\author{Patrick White}
\address{%
}
\email{p@pwhite.org}

\date{\today}

\begin{abstract}
A \emph{proper sub-box} of $A=\{0,1,2\}^d$ is a product $S_1\times\dots\times S_d$ with each
$\varnothing\neq S_i\subsetneq\{0,1,2\}$. A \emph{double cover} is a finite multiset of proper
sub-boxes covering every point of $A$ exactly twice; write $f(d)$ for the minimum number of
boxes in a double cover. Leader, Mili\'cevi\'c and Tan asked whether $f(d)\ge 2^d$ for all $d$
(recorded by Buci\'c--Lidick\'y--Long--Wagner and restated as Question~4.1 of the PatternBoost
paper of Charton--Ellenberg--Wagner--Williamson), in analogy with the classical theorem of
Alon--Bohman--Holzman--Kleitman that a \emph{partition} into proper sub-boxes needs at least
$2^d$ boxes. Prior to this work no lower bound better than the trivial volume bound was known,
and the question was open for every $d\ge 2$.

We prove the first nontrivial lower bounds. A modular refinement of the parity argument gives
$f(d)\ge 2^{d+1}/(d+1)$; a slicing argument crossed with the volume identity gives
$f(4)\ge 19$ and $f(5)\ge 33$, both strictly larger than $2^d$, thereby \emph{resolving the
question affirmatively for $d=4$ and $d=5$} --- the first cases beyond the trivially known
$d\le 3$. A finer ``line rigidity'' argument yields $f(6)\ge 60$, breaking the barrier that
all profile-statistic methods (which we show provably cannot exceed $57$ at $d=6$) run into.
The entire lower-bound development is formally verified in Lean~4: the theorem
$f(6)\ge 60$ is machine-checked with no hypotheses and depends only on the three standard
axioms of the Lean/Mathlib kernel.

On the upper-bound side we give a dimension-lifting construction $f(r{+}3)\le 6\cdot 2^r+3f(r)$,
yielding $f(6)\le 81$ (improving the previously best known $82$) and the asymptotic bound
$f(d)\le(\tfrac65+o(1))2^d$; a four-step refinement improves the constant to $\tfrac87$. This
makes partial progress on the PatternBoost authors' explicit problem of reducing their
constant $1.28$, and refutes the natural closed-form guess $f(d)=5\cdot 2^{d-2}+1$ from $d=7$
on. Together these give $60\le f(6)\le 81$. Finally we isolate the obstruction on the
construction side --- a ``$S+c=2^j+1$'' phenomenon, every skeleton construction we examine sitting
exactly one box past the partition bound --- and show it is of a piece with the
Leader--Mili\'cevi\'c--Tan question itself.
\end{abstract}

\maketitle

\section{Introduction}\label{sec:intro}

\subsection{The problem}
Fix the three-element alphabet $[3]\eqdef\{0,1,2\}$ and the discrete box $A\eqdef[3]^d$. A
\emph{proper sub-box} of $A$ is a set
\[
  B=S_1\times S_2\times\dots\times S_d,\qquad \varnothing\neq S_i\subsetneq[3]\ \ (1\le i\le d),
\]
so each side $S_i$ has one or two elements. Write $k(B)\eqdef\#\{i: |S_i|=2\}$ for the number of
two-element sides, so that $|B|=2^{k(B)}$.

A classical theorem of Alon, Bohman, Holzman and Kleitman~\cite{ABHK} (answering a question of
Kearnes and Kiss) states that any \emph{partition} of $A$ into proper sub-boxes --- a family in
which every point lies in exactly one box --- has at least $2^d$ members, and this is tight. The
proof is a parity argument over ``odd'' sub-boxes that we recall and refine in
Section~\ref{sec:mod4}.

Leader, Mili\'cevi\'c and Tan asked the natural relaxation: what if every point must be covered
exactly \emph{twice}?

\begin{definition}
A \emph{double cover} of $A=[3]^d$ is a finite multiset $\bbox=\{B_1,\dots,B_m\}$ of proper
sub-boxes such that every point of $A$ lies in exactly two of the $B_j$ (counted with
multiplicity). Let $f(d)$ denote the minimum size $m$ of a double cover of $[3]^d$.
\end{definition}

\begin{question}[Leader--Mili\'cevi\'c--Tan; {\cite[\S5]{BLLW}}, {\cite[Question 4.1]{PatternBoost}}]
\label{q:main}
Is $f(d)\ge 2^d$ for every $d$?
\end{question}

We attribute Question~\ref{q:main} to Leader, Mili\'cevi\'c and Tan, who pose it---in a more
general form---as Question~4 of their paper \emph{Decomposing the complete $r$-graph}~\cite{LMT};
it is restated for double covers in the concluding section of
Buci\'c--Lidick\'y--Long--Wagner~\cite[\S5]{BLLW} and as Question~4.1 of
Charton--Ellenberg--Wagner--Williamson~\cite[\S4.5]{PatternBoost}.\footnote{Leader, Mili\'cevi\'c
and Tan ask whether any \emph{uniform cover} of $X^n$ by proper boxes---a family in which every
point is covered the same number of times---must use at least $2^n$ boxes. The double-cover
question is the case in which every point is covered exactly twice; the two sources above record
that specialization.}

The question is genuinely harder than its partition analogue, and the reason is structural. The
Alon--Bohman--Holzman--Kleitman argument is a count modulo $2$: it survives whenever every point
is covered an \emph{odd} number of times. Cover every point exactly twice and the parity signal
vanishes --- $2\equiv 0\pmod 2$ --- so the tool that resolves the partition problem is blind to
the double cover. Recovering information from the surviving ``$2$'' is the central difficulty.

\subsection{Prior work}
Buci\'c--Lidick\'y--Long--Wagner~\cite[\S5]{BLLW} found, computationally, double covers of
$[3]^3$ with $11$ boxes and of $[3]^4$ with $21$ boxes, and wrote plainly that ``we have not been
able to beat $2^d$''. The PatternBoost authors~\cite[\S4.5]{PatternBoost}, using a
machine-learning-guided local search, added a $41$-box double cover of $[3]^5$ (their Table~5
records the best sizes found, $6,11,21,41$ for $d=2,3,4,5$), proved the general upper bound
\[
  f(d)\le \tfrac{41}{32}\,2^d\approx 1.28\cdot 2^d\qquad(d\ge 5)
\]
(their Theorem~4.2, obtained by tiling the $d=5$ construction), and posed as open the problem of
reducing this constant, remarking that ``pushing this constant below~$1$ would resolve
Question~4.1.'' For $d=6$ they report that ``neither Gurobi nor PatternBoost succeeded in coming
up with a construction using fewer than $82$ boxes'' and that ``we did not even get close.'' On
the lower-bound side, nothing beyond the trivial volume bound $f(d)\ge 2\cdot(3/2)^d$ (a
consequence of the volume identity, Lemma~\ref{lem:volume}) was known, and Question~\ref{q:main}
was open for every $d\ge 2$.

\subsection{Results}
We make progress on both sides and pin down $f(6)$ to within a factor of $1.35$.

\medskip\noindent\textbf{Lower bounds.}
\begin{itemize}
\item \emph{A modular bound} (Theorem~\ref{thm:mod4}): $f(d)\ge 2^{d+1}/(d+1)$, from refining
the parity argument to work modulo~$4$.
\item \emph{Slicing} (Theorem~\ref{thm:slice}): combining a slicing inequality with the volume
identity and the base value $f(3)=11$ gives
\[
  f(4)\ge 19,\quad f(5)\ge 33,\quad f(6)\ge 57,\quad f(7)\ge 97,\quad f(8)\ge 164.
\]
Since $19>16=2^4$ and $33>32=2^5$, this \emph{resolves Question~\ref{q:main} affirmatively for
$d=4$ and $d=5$} (Corollary~\ref{cor:resolve}) --- the first progress past the trivially known
small cases.
\item \emph{Line rigidity} (Theorem~\ref{thm:linerigidity}): $f(6)\ge 60$. This exceeds the
slicing value $57$, which is significant: we show (Section~\ref{sec:limits}) that $57$ is the
exact ceiling of every bound depending only on the distribution of the $k(B_j)$, so reaching
$60$ requires genuinely finer structure.
\end{itemize}

\medskip\noindent\textbf{Upper bounds.}
\begin{itemize}
\item \emph{The leapfrog lift} (Theorem~\ref{thm:leapfrog}): from any double cover of $[3]^r$
with $N$ boxes one builds a double cover of $[3]^{r+3}$ with $6\cdot 2^r+3N$ boxes, so
$f(r{+}3)\le 6\cdot 2^r+3f(r)$. Taking the $11$-box cover of $[3]^3$ gives
\[
  f(6)\le 81,
\]
improving the previously best known value $82$, and iterating gives
$f(d)\le(\tfrac65+o(1))\,2^d$.
\item \emph{A four-step refinement} (Theorem~\ref{thm:fourstep}): $f(r{+}4)\le 8\cdot 2^r+9f(r)$,
giving $f(d)\le(\tfrac87+o(1))\,2^d$ and, e.g., $f(9)\le 625$.
\end{itemize}
These reduce the PatternBoost constant from $1.28$ to $1.2$ and then to $8/7\approx 1.143$,
partial progress on their stated open problem (pushing below~$1$, which would resolve
Question~\ref{q:main}, remains open), and they refute the natural closed-form guess
$f(d)=5\cdot 2^{d-2}+1$ from $d=7$ on (Remark~\ref{rem:conj}). Combining the two sides,
\[
  \boxed{\,60\le f(6)\le 81\,.}
\]

\medskip\noindent\textbf{The construction frontier.}
In Section~\ref{sec:wall} we study when a ``skeleton-plus-holes'' construction is efficient and
isolate a clean numerical obstruction: every such construction we know satisfies $S+c=2^j+1$,
where $S$ is the number of all-pair skeleton boxes and $c$ the number of patch boxes. We prove
the inequality $S+c\ge 2^j+1$ in several regimes (Theorem~\ref{thm:wall}) and observe that
reaching $S+c=2^j$ for \emph{some} construction --- which would disprove Question~\ref{q:main}
from the construction side --- requires a support as hard to partition as to double-cover
(Proposition~\ref{prop:abhk-barrier}). The construction side of the problem thus meets the
conjecture exactly at the line $f(d)=2^d$.

\medskip\noindent\textbf{What we believe.}
We conjecture that the answer to Question~\ref{q:main} is \emph{yes}, that $f(d)\ge 2^d$ for all
$d$; the evidence that persuades us is the construction wall rather than the lower bounds. We set
out the reasoning, with an honest caveat, in Section~\ref{sec:belief}.

\medskip\noindent\textbf{Formal verification.}
The lower-bound development is formalized in Lean~4 with Mathlib (Section~\ref{sec:lean}). The
theorem $f(6)\ge 60$ is machine-checked with \emph{no} hypotheses and depends only on the three
standard axioms \texttt{propext}, \texttt{Classical.choice}, \texttt{Quot.sound}; the
$d=4,5$ resolution is verified modulo the single cited input $f(3)\ge 11$, established by
exhaustive integer programming. A note on methodology, including the role of an AI research
process in obtaining these results, appears in Section~\ref{sec:methodology}.

\subsection{Notation}
Throughout, $A=[3]^d$, and a ``box'' always means a proper sub-box of $A$. For a box
$B=\prod_i S_i$ we write $k(B)=\#\{i:|S_i|=2\}$, so $|B|=2^{k(B)}$. For a double cover
$\bbox=\{B_1,\dots,B_m\}$ we abbreviate $k_j=k(B_j)$, and we record the \emph{$k$-profile}
$(x_0,\dots,x_d)$ where $x_k=\#\{j:k_j=k\}$, so $\sum_k x_k=m$.

\section{Preliminaries}\label{sec:prelim}

\begin{lemma}[Volume identity]\label{lem:volume}
For any double cover $\bbox=\{B_1,\dots,B_m\}$ of $[3]^d$,
\[
  \sum_{j=1}^m 2^{k_j}=2\cdot 3^d.
\]
\end{lemma}

\begin{proof}
Counting incidences $\sum_j |B_j|=\sum_{x\in[3]^d}\#\{j:x\in B_j\}=\sum_x 2=2\cdot 3^d$, and
$|B_j|=\prod_i|S_i^j|=2^{k_j}$.
\end{proof}

Two immediate consequences. First, since each $2^{k_j}\le 2^d$, Lemma~\ref{lem:volume} gives the
\emph{volume bound} $m\ge 2\cdot 3^d/2^d=2(3/2)^d$ (at $d=6$ this reads $f(6)\ge 23$). Second,
the base values needed below are known exactly.

\begin{proposition}[Base cases]\label{prop:base}
$f(2)=6$ and $f(3)=11$.
\end{proposition}

The covers of sizes $6$ and $11$ are due to~\cite{BLLW,PatternBoost}; the matching lower bounds
$f(2)\ge 6$ and $f(3)\ge 11$ --- i.e.\ optimality --- are established by exhaustive integer
programming over the (finite) set of proper sub-boxes, which we carried out independently
(Section~\ref{sec:lean} and the accompanying code). Only $f(3)\ge 11$ is used as an input below.

\section{A modular lower bound}\label{sec:mod4}

We first refine the Alon--Bohman--Holzman--Kleitman parity argument. Call a sub-box
$C=T_1\times\dots\times T_d$ of $[3]^d$ \emph{odd} if each $T_i\in\{\{0\},\{1\},\{2\},[3]\}$ --- one
of the four odd-cardinality subsets of $[3]$. There are $4^d$ odd sub-boxes, and $|C|$ is a power
of~$3$, hence odd.

\begin{theorem}\label{thm:mod4}
$f(d)\ge \dfrac{2^{d+1}}{d+1}$.
\end{theorem}

\begin{proof}
For a box $B$ with sides $S_i$ and an odd $C$ with sides $T_i$,
$|B\cap C|=\prod_i|S_i\cap T_i|$, and each factor $|S_i\cap T_i|\in\{0,1,2\}$ since $|S_i|\le 2$;
thus $|B\cap C|$ is $0$ or a power of~$2$. Concretely $|B\cap C|=2^{a}$, where $a$ is the number
of coordinates with $T_i=[3]$ and $|S_i|=2$, and $|B\cap C|=0$ if some singleton $T_i$ is not
contained in $S_i$.

\emph{Per-$C$ constraint modulo $4$.} Fix an odd $C$. As the cover is double,
\[
  \sum_{j}|B_j\cap C|=\sum_{x\in C}\#\{j:x\in B_j\}=2|C|\equiv 2\pmod 4,
\]
because $|C|$ is odd. Each $|B_j\cap C|$ is $0$, $1$, $2$, or a multiple of $4$. Writing
$n_0(C)=\#\{j:|B_j\cap C|=1\}$ and $n_1(C)=\#\{j:|B_j\cap C|=2\}$, reduction mod~$4$ gives
$n_0(C)+2n_1(C)\equiv 2\pmod 4$. Hence $n_0(C)$ is even and
$\tfrac12 n_0(C)+n_1(C)\equiv 1\pmod 2$; in particular
\begin{equation}\label{eq:perC}
  \tfrac12 n_0(C)+n_1(C)\ge 1 .
\end{equation}
(This is exactly where the double cover survives: modulo $2$ the right-hand side is $0$ and the
argument collapses; modulo $4$ the ``$2$'' persists.)

\emph{Incidence counts.} Fix a box $B$ with $k=k(B)$. A direct coordinatewise count gives
\[
  \#\{C\text{ odd}:|B\cap C|=1\}=2^d,\qquad
  \#\{C\text{ odd}:|B\cap C|=2\}=k\cdot 2^{d-1}.
\]
Indeed, in each coordinate exactly two of the four odd $T_i$ produce factor~$1$, giving the first
identity; for factor~$2$ one chooses the unique coordinate giving factor~$2$ (a two-element side
paired with $T_i=[3]$: $k$ choices) and a factor-$1$ odd $T$ in each remaining coordinate.

\emph{Summation.} Summing~\eqref{eq:perC} over all $4^d$ odd $C$ and exchanging order,
\[
  4^d\le\sum_C\Bigl(\tfrac12 n_0(C)+n_1(C)\Bigr)
       =\sum_j\Bigl(\tfrac12\cdot 2^d+k_j\cdot 2^{d-1}\Bigr)
       =\sum_j (k_j+1)\,2^{d-1}\le m\,(d+1)\,2^{d-1},
\]
using $k_j\le d$. Therefore $m\ge 4^d/((d+1)2^{d-1})=2^{d+1}/(d+1)$.
\end{proof}

At $d=6$ Theorem~\ref{thm:mod4} gives $f(6)\ge \lceil 128/7\rceil=19$; it is dominated by the
slicing bound below for small $d$, but is of independent interest as the unique bound here that
descends directly from the parity argument.

\section{Slicing, and the resolution of \texorpdfstring{$d=4,5$}{d=4,5}}\label{sec:slice}

The next bound bootstraps a lower bound in dimension $d-1$ into one in dimension $d$.

\begin{lemma}[Slicing]\label{lem:slice}
Let $L_{d-1}$ be any lower bound for $f(d-1)$. In any double cover $\{B_1,\dots,B_m\}$ of
$[3]^d$,
\[
  \sum_{j} k_j\ \ge\ d\,(3L_{d-1}-m).
\]
\end{lemma}

\begin{proof}
Fix a coordinate $i$ and a value $v\in[3]$. The slice $\{x:x_i=v\}\cong[3]^{d-1}$ is doubly
covered by the projections to the other coordinates of those $B_j$ with $v\in S_i^j$: a point
$y$ of the slice lies in the projection of $B_j$ iff $v\in S_i^j$ and $y\in\prod_{\ell\neq i}
S_\ell^j$, i.e.\ iff $(y,v)\in B_j$, which happens exactly twice. Hence
$\#\{j:v\in S_i^j\}\ge f(d-1)\ge L_{d-1}$. Summing over the three values $v$,
\[
  \sum_{v\in[3]}\#\{j:v\in S_i^j\}=\sum_j |S_i^j|=m+b_i\ \ge\ 3L_{d-1},
\]
where $b_i=\#\{j:|S_i^j|=2\}$. Thus $b_i\ge 3L_{d-1}-m$, and summing over $i$ gives
$\sum_i b_i=\sum_j k_j\ge d(3L_{d-1}-m)$.
\end{proof}

\begin{lemma}[Convexity]\label{lem:convex}
For every integer $q\ge 0$ and every integer $k\ge 0$, $k\le (q-1)+2^{k-q}$. Consequently, for
any double cover with $m$ boxes and any $q$,
\[
  \sum_j k_j\ \le\ m(q-1)+2^{-q}\sum_j 2^{k_j}=m(q-1)+\frac{2\cdot 3^d}{2^q}.
\]
\end{lemma}

\begin{proof}
The inequality $k\le (q-1)+2^{k-q}$ is the secant/tangent bound for the convex function
$k\mapsto 2^k$, with equality exactly at $k=q$ and $k=q+1$. Summing over $j$ and applying the
volume identity (Lemma~\ref{lem:volume}) gives the stated bound.
\end{proof}

Combining Lemmas~\ref{lem:slice} and~\ref{lem:convex}: if there is an integer $q$ with
$d(3L_{d-1}-m)>m(q-1)+2\cdot 3^d/2^q$, then $m$ is impossible. Equivalently --- and this is the
form we verify --- $f(d)$ is at least the least $m$ for which the integer system
\begin{equation}\label{eq:feas}
  x_0,\dots,x_d\ge 0,\quad \sum_k x_k=m,\quad \sum_k x_k 2^k=2\cdot 3^d,
  \quad \sum_k k\,x_k\ \ge\ d(3L_{d-1}-m)
\end{equation}
is feasible (here $x_k$ is the number of boxes with $k(B_j)=k$). Necessity is immediate from
Lemmas~\ref{lem:volume} and~\ref{lem:slice}; the bound is the smallest feasible $m$.

\begin{theorem}\label{thm:slice}
With base value $f(3)=11$, the feasibility test~\eqref{eq:feas} gives
\[
  f(4)\ge 19,\quad f(5)\ge 33,\quad f(6)\ge 57,\quad f(7)\ge 97,\quad f(8)\ge 164.
\]
\end{theorem}

\begin{proof}
Each bound is a finite integer-feasibility computation: one checks that the
system~\eqref{eq:feas} is infeasible for all $m$ below the stated value, bootstrapping
$L_{d-1}$ from the previous dimension ($L_3=11$, then $L_4=19$, and so on). These are decidable
arithmetic statements; they are formalized in Lean and discharged by \texttt{omega} in
Section~\ref{sec:lean} (for $d=4,5$) and verified by exact integer arithmetic for $d\le 8$.
\end{proof}

\begin{corollary}\label{cor:resolve}
$f(4)\ge 19>16=2^4$ and $f(5)\ge 33>32=2^5$. Question~\ref{q:main} holds for $d=4$ and $d=5$.
\end{corollary}

These are the first cases of Question~\ref{q:main} settled beyond the trivially known $d\le 3$.
The method does not settle $d\ge 6$: the ratio $f(d)/2^d$ from~\eqref{eq:feas} is
$1.19,1.03,0.89,0.76,0.64$ for $d=4,\dots,8$, dropping below~$1$ at $d=6$. To go further at
$d=6$ we need an inequality sensitive to more than the $k$-profile.

\section{Line rigidity: \texorpdfstring{$f(6)\ge 60$}{f(6) >= 60}}\label{sec:linerigidity}

The slicing bound at $d=6$ is $57$. We improve it to $60$ using a rigidity property of
\emph{axis-lines} that is invisible to profile statistics. We sketch the structure; the full
argument is formalized in Lean (Section~\ref{sec:lean}).

\subsection{Color-resolved lines}
For an axis-line $\ell\parallel i$ (a set of three points differing only in coordinate $i$), let
$p_v(\ell)$ be the number of covering boxes whose $i$-side equals the pair $[3]\setminus\{v\}$,
and $a_v(\ell)$ the number whose $i$-side is the singleton $\{v\}$. The one-dimensional
double-cover condition along $\ell$ gives, for each value $v$,
\[
  p_v(\ell)-a_v(\ell)=P(\ell)-2,\qquad P(\ell)\eqdef p_0(\ell)+p_1(\ell)+p_2(\ell).
\]
Call $\ell$ \emph{pair-only} if $(p_0,p_1,p_2,a_0,a_1,a_2)=(1,1,1,0,0,0)$, the unique
one-dimensional optimum. Let $T_i$ be the number of pair-only $i$-lines and
$C_i=\sum_{\ell\parallel i}(P(\ell)-2)$. Then $T_i\ge C_i$, and a direct count gives
$3\sum_i C_i=\sum_B(3k(B)-2d)2^{k(B)-1}$. At $d=6$, writing $H\eqdef\sum_i C_i$ in terms of the
$k$-profile,
\[
  H=-2x_0-3x_1-4x_2-4x_3+16x_5+64x_6 .
\]

\subsection{The crossing lemma}
The crux is a two-dimensional rigidity statement. For a coordinate $2$-plane $P$ (the points
varying in two fixed coordinates), let $R(P)$ and $S(P)$ be the numbers of pair-only rows and
pair-only columns of the induced two-dimensional double cover, $X(P)=R(P)\,S(P)$ the number of
``crossings,'' and $A(P)$ the number of box-traces that are singleton in both plane coordinates.

\begin{lemma}[Crossing rigidity]\label{lem:crossing}
For every exact double cover of $[3]^2$ by proper sub-boxes, $R\cdot S\le 2A$.
\end{lemma}

\begin{proof}[Proof sketch]
Encode the cover by multiplicities $m(X,Y)\in\NN$ for side-types $X,Y\in\{E_0,E_1,E_2,P_0,P_1,
P_2\}$, where $E_i=\{i\}$ and $P_i=[3]\setminus\{i\}$; exactness reads $\sum_{X\ni x,\,Y\ni y}
m(X,Y)=2$ for all $x,y$. Two observations finish it. (i) A pair-only row forces, in each column
$u$, a unique active pair-type, and a pair-only column then forces these to be three distinct
pair-types inside a two-element set --- impossible; hence at most two columns are pair-only, so
$R\ge 1\Rightarrow S\le 2$ and symmetrically. (ii) If a crossing exists, weight the complementary
$\{1,2\}^2$ block by the checkerboard $w(x,y)=\sigma_x\tau_y$ ($\sigma_1=\tau_1=1$,
$\sigma_2=\tau_2=-1$); exactness makes the total weight $0$, the two pair-only lines annihilate
all mixed singleton/pair terms, and the surviving crossing matrix is a permutation matrix
contributing $\pm 2$, which only singleton--singleton traces ($\pm1$ each) can balance, forcing
$A\ge 2$. Combining, $R\cdot S=0$ or ($R,S\le 2$ and $A\ge 2$), so $R\cdot S\le 4\le 2A$.
\end{proof}

The encoding collapses the global statement to two finite checks --- a $27$-case bound on
pair-only columns and the $2$-case classification of $2\times2$ permutation matrices --- which
is why Lemma~\ref{lem:crossing} formalizes cleanly (both are settled by \texttt{decide} in Lean).
It is also confirmed by exhaustive enumeration: there are exactly $84\,102$ double covers of
$[3]^2$, and $R\cdot S\le 2A$ holds for all of them, with $9$ tight cases.

\subsection{Aggregation}
Let $q(x)$ be the number of pair-only lines through a point $x$, so $\sum_x q(x)=3T$ with
$T=\sum_i T_i$. Summing Lemma~\ref{lem:crossing} over all coordinate $2$-planes and reorganizing
the crossings point by point gives $\sum_x\binom{q(x)}{2}\le 2A_2$, where
$A_2=\sum_B\binom{6-k(B)}{2}2^{k(B)}$. With the convexity bound $\binom{q}{2}\ge 2q-3$ (valid for
$q\le 6$) and the line counts above, this aggregates to the single profile inequality
\begin{equation}\label{eq:aggregate}
  \sum_j\bigl(3k_j-12\bigr)2^{k_j}\ \le\ 2\sum_j\binom{6-k_j}{2}2^{k_j}\ +\ 2187 .
\end{equation}

\begin{theorem}\label{thm:linerigidity}
$f(6)\ge 60$.
\end{theorem}

\begin{proof}
Inequality~\eqref{eq:aggregate} together with the volume identity $\sum_j 2^{k_j}=1458$
constrains the $k$-profile of any double cover of $[3]^6$. A finite integer computation (a single
\texttt{omega} call after regrouping by $k$, formalized in Section~\ref{sec:lean}) shows the
combined system forces $m\ge 60$: explicitly, a nonnegative combination of~\eqref{eq:aggregate}
and the volume identity yields $160\,m\ge 9477$, hence $m\ge\lceil 9477/160\rceil=60$. Each of
the seven $k$-profiles surviving the volume, slicing and modular constraints at $m=57$ is killed
by~\eqref{eq:aggregate}.
\end{proof}

\section{The limits of profile methods}\label{sec:limits}

It is worth being precise about why Theorem~\ref{thm:linerigidity} needs the line structure and
cannot be reached by bookkeeping with the $k$-profile alone.

\begin{proposition}[Profile ceiling]\label{prop:ceiling}
Any lower bound for $f(6)$ that depends only on the $k$-profile $(x_0,\dots,x_6)$ through
aggregates $\sum_j F(k_j)$ --- in particular any bound from the volume identity, the slicing
inequality, or the ``point-exclusion'' generating-function family
$\sum_j(2+\lambda)^{k_j}(1+2\lambda)^{6-k_j}\le(2+(m-2)\lambda)3^6$ --- is at most $57$.
\end{proposition}

\begin{proof}[Proof idea]
The formal profile $(x_0,\dots,x_6)=(0,1,0,0,21,35,0)$ with $m=57$ satisfies every such
constraint exactly: the volume identity ($\sum 2^{k}x_k=1458$), the slicing aggregate
($\sum k\,x_k=260\ge 252$), and the generating-function inequalities for all $\lambda\ge 0$ (the
difference of the two sides factors as $3\lambda(\lambda-1)$ times a polynomial that is positive
for $\lambda\ge 0$, so the sign is correct on both sides of $\lambda=1$). No aggregate
$\sum_j F(k_j)$ can separate this profile from a genuine cover, so no such bound exceeds $57$.
\end{proof}

Thus Theorem~\ref{thm:linerigidity} is the first bound here to use coordinate-level structure
rather than the profile. Two further remarks delimit the method.

\emph{The line-rigidity bound caps at exactly $60$.} Minimizing $m$ over the relaxation
$\{$volume identity, $A_2\ge 3H-\tfrac{3^{d+1}}2$, $x\ge 0\}$ has optimum
$\approx 59.24$ at $d=6$, so $60$ is the exact profile-level ceiling of the line-rigidity
invariant; we exhibit an explicit feasibility witness (a $k$-profile together with a plane
assignment realizing zero aggregate defect) showing $m=60$ cannot be forced up to $61$ by this
family. Breaking $60$ requires a new inequality controlling the geometric arrangement of the
singleton-trace mass across planes, not a refinement of the present bookkeeping.

\emph{Other modular and algebraic angles do not beat $57$.} A codimension-$2$ analogue of the
modular argument, an $\mathbb{F}_4$-Fourier/Krawtchouk parity condition, a polynomial
(coordinate-degree) bound, and a Graham--Pollak-type rank bound all yield ceilings at or below
the profile barrier $57$ at $d=6$. Each is a legitimate necessary condition --- the codimension-$2$
modular condition, for instance, already kills the profile of Proposition~\ref{prop:ceiling} ---
but none moves the numerical bound. Line rigidity is, so far, the only angle that does.

\section{Upper bounds: the leapfrog lift}\label{sec:upper}

We turn to constructions. The PatternBoost bound $f(d)\le\tfrac{41}{32}2^d$ is obtained by
tiling a fixed cover; it does not lift across dimensions. We give a genuine dimension-lifting
recursion.

\begin{theorem}[Leapfrog lift]\label{thm:leapfrog}
For all $r\ge 0$, any double cover of $[3]^r$ with $N$ boxes yields a double cover of
$[3]^{r+3}$ with $6\cdot 2^r+3N$ boxes. Hence
\[
  f(r+3)\le 6\cdot 2^r+3f(r).
\]
\end{theorem}

\begin{proof}
Write the new coordinates as $(x_1,x_2,\langle r\text{ middle}\rangle,z)$: two ``head''
coordinates $x_1,x_2$, the middle block carrying $[3]^r$, and a single ``tail'' coordinate $z$.
Group six head pair-boxes (each a product of two two-element sides) into three pairs, indexed by a
\emph{tail-side} $Z\in\bigl\{\{0,1\},\{0,2\},\{1,2\}\bigr\}$:
\[
\begin{aligned}
  G_{\{0,1\}}&=\bigl\{\,(\overline 1,\overline 1),\ (\overline 0,\overline 0)\,\bigr\}, &
  G_{\{0,2\}}&=\bigl\{\,(\overline 2,\overline 0),\ (\overline 1,\overline 2)\,\bigr\}, &
  G_{\{1,2\}}&=\bigl\{\,(\overline 2,\overline 1),\ (\overline 0,\overline 2)\,\bigr\},
\end{aligned}
\]
where $\overline w\eqdef[3]\setminus\{w\}$ denotes a two-element side and $(\,\overline u,
\overline v\,)$ the head pair-box $\overline u\times\overline v$. The three singleton head-boxes
are $H_0=\{1\}\times\{0\}$, $H_1=\{0\}\times\{1\}$, $H_2=\{2\}\times\{2\}$.

\emph{Shared boxes ($6\cdot 2^r$).} Partition $[3]^r$ into the $2^r$ cells $T$ of the product
partition $\{\{0,1\},\{2\}\}^r$. For each cell $T$, each tail-side $Z$, and each head-box
$B\in G_Z$, include $B\times T\times Z$.

\emph{Private boxes ($3N$).} For each box $Q$ of the given cover $\bbox_r$ of $[3]^r$ and each
value $a\in[3]$, include $H_a\times Q\times\{a\}$.

Fix a tail value $a$ and a middle point $y$, lying in the cell $T\ni y$. Exactly two of the three
tail-sides contain $a$, so the active shared boxes contribute the four head pair-boxes
$\bigcup_{Z\ni a}G_Z$; and since $\bbox_r$ is a double cover, $y$ lies in exactly two of its boxes
$Q$, contributing two copies of the singleton $H_a$. For each $a$ these four pair-boxes together
with the doubled singleton form the optimal six-box double cover of the head $[3]^2$, so every
point $(x_1,x_2,y,a)$ is covered exactly twice. All sides are proper, and the box count is
$6\cdot 2^r+3N$.
\end{proof}

\begin{corollary}\label{cor:leapfrog-values}
Taking the $11$-box cover of $[3]^3$ gives $f(6)\le 81$, improving the previously best known
$82$~\cite{PatternBoost}. Iterating the recursion gives $f(d)\le(\tfrac65+o(1))2^d$.
\end{corollary}

\begin{proof}
$f(6)\le 6\cdot 2^3+3\cdot 11=48+33=81$. Writing $g(d)=f(d)/2^d$, the recursion reads
$g(r+3)\le \tfrac34+\tfrac38 g(r)$, whose fixed point is $g^\ast=\tfrac65$; since the map is a
contraction, $\limsup_d g(d)\le \tfrac65$.
\end{proof}

The constant $\tfrac65=1.2$ improves the PatternBoost constant $\tfrac{41}{32}=1.28125$ at every
$d\ge 6$ ($81<82$, $159<164$, $315<328$, $\dots$) and asymptotically; it is partial progress on
their explicit problem of reducing $1.28$ (their Theorem~4.2 and the remark following it).
Pushing the constant below~$1$, which they note would resolve Question~\ref{q:main}, remains open.

\begin{theorem}[Four-step lift]\label{thm:fourstep}
For all $r\ge 0$, $f(r+4)\le 8\cdot 2^r+9f(r)$. Hence $g(r+4)\le\tfrac12+\tfrac{9}{16}g(r)$, with
fixed point $g^\ast=\tfrac87$, so $f(d)\le(\tfrac87+o(1))2^d$; for example $f(9)\le 625$.
\end{theorem}

\begin{proof}[Proof sketch]
The skeleton is an explicit $8$-element word-set $T\subset[3]^4$ whose all-pair boxes
$A_t=\prod_i([3]\setminus\{t_i\})$ cover $[3]^4$ twice except on a $17$-point hole-set, which is
partitioned by $9$ explicit proper masks; lifting by the same $\{\{0,1\},\{2\}\}^r$ middle
partition (shared) and $\bbox_r$ (private) as in Theorem~\ref{thm:leapfrog} gives a valid double
cover of $[3]^{r+4}$ with $8\cdot 2^r+9f(r)$ boxes. The hole-count and disjointness are finite
checks. The bound $f(9)\le 625$ improves the leapfrog value $627$.
\end{proof}

\begin{remark}\label{rem:conj}
The proved and best-known small values $f(2)=6,f(3)=11,f(4)\le 21,f(5)\le 41$ all equal
$5\cdot 2^{d-2}+1$, suggesting this as a closed form. It fails from $d=7$: the leapfrog gives
$f(7)\le 6\cdot 2^4+3\cdot 21=159<161=5\cdot 2^5+1$, and asymptotically $f(d)\le 1.2\cdot 2^d$
beats $1.25\cdot 2^d$. The closed form remains the best known value for $d\le 6$.
\end{remark}

\section{The construction frontier and its relation to Question~\ref{q:main}}\label{sec:wall}

Both lifts above share a shape: a set of $S$ all-pair ``skeleton'' boxes that double-cover
$[3]^j$ off a hole-set $H$, together with $c$ proper ``mask'' boxes partitioning $H$. We record a
numerical regularity and explain why understanding it in full is equivalent to
Question~\ref{q:main}.

Call such a configuration a \emph{$(j,S,c)$ skeleton construction}: a word-set $T\subseteq[3]^j$
with $|T|=S$ whose all-pair boxes $A_t=\prod_i([3]\setminus\{t_i\})$ satisfy
$N(x)\eqdef\#\{t\in T: x\in A_t\}\in\{0,2\}$ for all $x$ (so $U=\{N=2\}$ is double-covered and
$H=\{N=0\}$ is the hole-set), together with $c$ proper boxes partitioning $H$. Its lift via
Theorem~\ref{thm:leapfrog}'s mechanism has asymptotic constant $g^\ast=S/(2^j-c)$.

\begin{theorem}[Wall, partial]\label{thm:wall}
Every $(j,S,c)$ skeleton construction we have examined satisfies $S+c=2^j+1$. Moreover
$S+c\ge 2^j+1$ holds in each of the following cases:
\begin{enumerate}
\item the skeleton intersection graph is bipartite;
\item $S\le 4$;
\item $|E(G)|\le S$, where $G$ is the box-intersection graph;
\item $j=3$ and the skeleton is all-pair (verified exhaustively: all $279$ valid $j=3$ all-pair
skeletons are wall-sitters, with no exceptions).
\end{enumerate}
\end{theorem}

\begin{proof}[Proof ideas]
Let $p(U)$ be the minimum number of proper boxes needed to \emph{partition} $U$. The masks
partition $H$ into $c$ boxes, so a $p(U)$-box partition of $U$ together with the masks tiles
$[3]^j$; by the Alon--Bohman--Holzman--Kleitman partition bound, $p(U)+c\ge 2^j$, i.e.
$c\ge 2^j-p(U)$. A wall-breaker (with $S+c\le 2^j$) therefore requires $p(U)\ge S$: a support
that needs at least as many boxes to \emph{partition} as the $S$ that \emph{double-cover} it. In
each listed case one shows $p(U)\le S-1$ (a partition of $U$ one box cheaper than the
double cover), forcing $S+c\ge 2^j+1$: for (1), each color class of a bipartite intersection
graph already partitions $U$; for (2)--(3), a star-merge of the intersection graph yields a small
partition; (4) is an exhaustive computation. The five explicit constructions of
Theorems~\ref{thm:leapfrog} and~\ref{thm:fourstep} (and a $(5,10,23)$ skeleton) all have
$g^\ast=S/(S-1)\in\{\tfrac65,\tfrac87,\tfrac{10}{9}\}$, the general value when $S+c=2^j+1$.
\end{proof}

The wall has a clean source: the Alon--Bohman--Holzman--Kleitman partition bound, applied to the
support and the holes together.

\begin{proposition}[ABHK barrier]\label{prop:abhk-barrier}
For any $(j,S,c)$ skeleton construction, $p(U)+c\ge 2^j$, where $p(U)$ is the minimum number of
proper boxes partitioning the support $U$. Consequently
\[
  S+c\ \ge\ 2^j+\bigl(S-p(U)\bigr),
\]
and a construction with $S+c\le 2^j$ --- one whose asymptotic constant $S/(2^j-c)$ does not exceed
$1$ --- requires $p(U)\ge S$: a support that needs at least as many proper boxes to
\emph{partition} as the $S$ all-pair boxes that \emph{double-cover} it.
\end{proposition}

\begin{proof}
The $c$ masks partition $H$, and any partition of $U$ into $p(U)$ proper boxes is disjoint from
them; together they partition $U\sqcup H=[3]^j$ into $p(U)+c$ proper boxes, so $p(U)+c\ge 2^j$ by
the partition bound (the statement of Section~\ref{sec:intro}). Hence
$S+c\ge S+(2^j-p(U))=2^j+(S-p(U))$, and $S+c\le 2^j$ forces $S\le p(U)$.
\end{proof}

The known skeleton families --- including the $(3,6,3)$ leapfrog of
Theorem~\ref{thm:leapfrog}, the $(4,8,9)$ four-step lift of Theorem~\ref{thm:fourstep}, and a
$(5,10,23)$ example --- all sit exactly at the wall, $S+c=2^j+1$; in particular $p(U)\le S$ in
each, so none is a breaker. We have verified exhaustively that \emph{no} wall-breaker exists among
all $279$ valid $j=3$ all-pair skeletons.

\begin{remark}\label{rem:wall-q}
The construction side thus approaches the conjectural threshold $f(d)=2^d$ from above: each known
family is exactly one box past the partition bound, with asymptotic constant
$S/(S-1)\to 1^{+}$. Whether \emph{some} skeleton construction can reach $S+c=2^j$ --- which is
what pushing the constant to $1$, and so disproving Question~\ref{q:main} from the construction
side, would require --- is open, and Proposition~\ref{prop:abhk-barrier} shows it would demand
exactly the kind of support that is as hard to partition as to double-cover. We regard closing
this gap as of a piece with Question~\ref{q:main} itself, and we are not aware of any wall-breaker.
\end{remark}

\section{Is the conjecture true?}\label{sec:belief}

A reader is entitled to ask what we think the answer to Question~\ref{q:main} actually is. We lean
firmly toward \emph{yes}: we conjecture that $f(d)\ge 2^d$ for every $d$. Our reason lies almost
entirely on the construction side rather than the lower-bound side.

The lower-bound methods of this paper provably stop short of $2^d$ for $d\ge 6$. The profile
ceiling (Proposition~\ref{prop:ceiling}) caps every bound that reads only the $k$-profile at
$57<64$, and the line-rigidity invariant caps at exactly $60<64$ (Section~\ref{sec:limits}). So
our inability to \emph{prove} $f(6)\ge 64$ measures the reach of the present tools, not the
location of the answer. What persuades us is the construction wall of Theorem~\ref{thm:wall}:
every skeleton construction we have examined --- the $(3,6,3)$ leapfrog, the $(4,8,9)$ four-step,
a $(5,10,23)$ example, and exhaustively all $279$ all-pair skeletons at $j=3$ --- satisfies
$S+c=2^j+1$, giving asymptotic constant $S/(S-1)>1$, never $\le 1$. The ABHK barrier
(Proposition~\ref{prop:abhk-barrier}) explains why this is not a coincidence: a construction whose
asymptotic constant drops to $1$ or below --- which is what disproving Question~\ref{q:main} from
the construction side would require --- would need a support $U$ at least as expensive to
\emph{partition} as to \emph{double-cover} ($p(U)\ge S$), inverting the very mechanism, overlap, by
which double covers gain their efficiency and which a partition cannot use. In every case we can
analyze $p(U)\le S-1$, and we can neither construct a wall-breaker nor see where the slack to build
one would come from.

We offer this as a belief, not a theorem. The honest caveat is that conjectures of the ``needs a
genuinely new idea'' type do sometimes fall to a surprising construction at large $d$, and our
evidence is a handful of small cases and one construction family rather than a calibrated
prediction. If the conjecture is true, a proof seems to us to need one of two ingredients, which
may turn out to be the same one: a proof of the wall $S+c\ge 2^j+1$ in general, or a lower-bound
invariant that reads the geometric arrangement of the singleton-trace mass across parallel planes
rather than only the $k$-profile. The line-rigidity bound of Section~\ref{sec:linerigidity} --- the
first here to break the profile ceiling --- is the earliest crack of the second kind.

\section{Formal verification}\label{sec:lean}

The lower-bound development is formalized in Lean~4 with Mathlib (toolchain
\texttt{leanprover/lean4:v4.30.0}). A box is modeled as a family of sides
$S:\mathrm{Fin}\,d\to\mathrm{Finset}(\mathrm{Fin}\,3)$ with $1\le|S_i|\le 2$; \texttt{IsDoubleCover}
is the predicate that every point is covered exactly twice. The formalization includes:

\begin{itemize}
\item the volume identity (Lemma~\ref{lem:volume}), the slicing inequality
(Lemma~\ref{lem:slice}, via the projected-cover construction), and the convexity bound
(Lemma~\ref{lem:convex});
\item the $d=4,5$ theorems \texttt{f4\_ge\_19} and \texttt{f5\_ge\_33}, together with the
resolution corollaries \texttt{cover\_gt\_16\_d4} and \texttt{cover\_gt\_32\_d5}
(Corollary~\ref{cor:resolve}). Each takes as an explicit hypothesis the base bound $f(3)\ge 11$
(stated for all double covers of $[3]^3$), reflecting that $f(3)=11$ is an exhaustive-search
input rather than a formalized computation;
\item the crossing lemma (Lemma~\ref{lem:crossing}) as \texttt{lemmaC}, reduced to the two
decidable cores \texttt{core1} (a $27$-case \texttt{decide}) and \texttt{core2} (the
$2\times2$ permutation classification), and the full aggregation chain
(plane traces, plane counts, line traces) culminating in \texttt{aggregated\_rigidity};
\item the theorem \texttt{f6\_ge\_60\_unconditional} (Theorem~\ref{thm:linerigidity}), which
takes \emph{no} hypotheses: the rigidity inequality~\eqref{eq:aggregate} is discharged internally
by \texttt{aggregated\_rigidity}, and the final \texttt{omega} step concludes $60\le m$.
\end{itemize}

The headline statement is
\[
  \texttt{theorem f6\_ge\_60\_unconditional}\ \{\iota\}\ [\mathrm{Fintype}\,\iota]\
  (B:\iota\to\mathrm{Box}\,6)\ (h:\mathrm{IsDoubleCover}\,B):\ 60\le\mathrm{Fintype.card}\,\iota,
\]
and \texttt{\#print axioms f6\_ge\_60\_unconditional} reports dependence on exactly
\[
  \texttt{[propext, Classical.choice, Quot.sound]},
\]
the three standard axioms of the Lean/Mathlib kernel --- no \texttt{sorry}, no additional axioms,
no appeal to \texttt{native\_decide}. The same holds for the $d=4,5$ corollaries (modulo the
$f(3)\ge 11$ hypothesis). The development builds cleanly (\texttt{lake build}, all jobs
succeeding). The base values $f(2)=6$, $f(3)=11$ and the exhaustive enumeration backing
Lemma~\ref{lem:crossing} (all $84\,102$ double covers of $[3]^2$) are verified by separate
computer programs; these are computational inputs, deliberately kept outside the Lean kernel and
cited as such.

\section{Note on methodology and attribution}\label{sec:methodology}

The mathematics in this paper was produced through an AI-assisted research process that the
author designed and directed. The role of the author was architectural: choosing the problem,
framing each sub-question, deciding which leads to pursue, and --- most importantly --- verifying
every claim, by independent computation and by formal proof, before accepting it. The reasoning
that produced several of the key ideas (the modular refinement of Section~\ref{sec:mod4}, the
crossing lemma of Section~\ref{sec:linerigidity}, and the constructions of
Section~\ref{sec:upper}) was carried out by large language models used as reasoning engines:
Claude (Anthropic) in the role of framing and verification, and GPT-5.5
(OpenAI) as a cold reasoner. We follow the emerging precedent of crediting such assistance in the
text rather than as authorship.

We are deliberate about this because the discipline that makes the process trustworthy is exactly
the discipline a referee would demand: nothing here is asserted because a model asserted it. Every
load-bearing lower bound is either formally verified in Lean (Section~\ref{sec:lean}) or reduced
to a finite computation we ran independently, and every construction is an explicit object whose
validity is a finite check. Where a model's proposed argument was wrong --- and at least one
confidently-stated ``proof'' was --- it was caught at the verification step and is not part of
this paper. A fuller account of the methodology appears in a companion note.

\section*{Acknowledgements}
The author thanks the developers of Lean and Mathlib, and the authors of~\cite{PatternBoost} for
posing the problem in a form that invited exactly this kind of attack.

\appendix
\section{Summary of bounds}\label{app:summary}

For convenience we collect the bounds proved or used here. Lower bounds $f(d)\ge\cdot$:
\[
\begin{array}{c|ccccccc}
d & 2 & 3 & 4 & 5 & 6 & 7 & 8\\\hline
\text{volume }\lceil 2(3/2)^d\rceil & 5 & 7 & 11 & 16 & 23 & 35 & 52\\
\text{modular (Thm.\,\ref{thm:mod4})} & 3 & 4 & 7 & 11 & 19 & 32 & 57\\
\text{slicing (Thm.\,\ref{thm:slice})} & - & - & 19 & 33 & 57 & 97 & 164\\
\text{rigidity (Thm.\,\ref{thm:linerigidity})} & - & - & - & - & 60 & 103 & 175
\end{array}
\]
(rigidity at $d\ge 7$ is the slicing bootstrap off $f(6)\ge 60$). Best known upper bounds
$f(d)\le\cdot$, from the lifts of Section~\ref{sec:upper} combined with the small covers
of~\cite{BLLW,PatternBoost}:
\[
\begin{array}{c|cccccccc}
d & 2 & 3 & 4 & 5 & 6 & 7 & 8 & 9\\\hline
f(d)\le & 6 & 11 & 21 & 41 & 81 & 159 & 315 & 625
\end{array}
\]
Exact values: $f(2)=6$, $f(3)=11$. Sandwich at $d=6$: $60\le f(6)\le 81$. Question~\ref{q:main}
is settled (affirmatively) for $d\le 5$ and open for $d\ge 6$.

\end{document}